\documentclass[A4, 11pt]{amsart}
\usepackage{amssymb, enumitem, color, tikz, hyperref, cleveref}
\usepackage[T1]{fontenc}
\usepackage{mathptmx}
\usepackage{todonotes, comment}
\usepackage{array}
\usepackage{multirow}
\usepackage{graphicx}

\newtheorem{theorem}{Theorem}
\numberwithin{theorem}{section}

\newtheorem{lemma}[theorem]{Lemma}
\newtheorem{corollary}[theorem]{Corollary}

\newtheorem{claim}[theorem]{Claim}
\theoremstyle{definition}
\newtheorem{definition}[theorem]{Definition}
\newtheorem{example}[theorem]{Example}
\newtheorem{remark}[theorem]{Remark}

\crefname{question}{question}{questions}

\DeclareMathOperator{\reg}{\operatorname{reg}}
\DeclareMathOperator{\codim}{\operatorname{codim}}
\DeclareMathOperator{\conn}{\operatorname{conn}}

\DeclareMathOperator{\PP}{\mathbb{P}}
\DeclareMathOperator{\ZZ}{\mathbb{Z}}
\DeclareMathOperator{\im}{\operatorname{im}}
\DeclareMathOperator{\rank}{\operatorname{rank}}

\title[Eisenbud--Goto for Stanley--Reisner ideals and simplicial complexes]{An Eisenbud--Goto type inequality for Stanley--Reisner ideals and simplicial complexes}
\author{Jaewoo Jung \and Jinha Kim \and Minki Kim \and Yeongrak Kim}

\address[Jaewoo Jung]{Center for Complex Geometry, Institute for Basic Science (IBS), Daejeon, South Korea}
\email{jaewoojung@ibs.re.kr}

\address[Jinha Kim]{Discrete Mathematics Group, Institute for Basic Science (IBS), Daejeon, South Korea}
\email{jinhakim@ibs.re.kr}

\address[Minki Kim]{Division of Liberal Arts and Sciences, Gwangju Institute of Science and Technology, Gwangju, South Korea}
\email{minkikim@gist.ac.kr}

\address[Yeongrak Kim]{Department of Mathematics, Pusan National University,  
Busan, South Korea
}

\address{Institute of Mathematical Science, Pusan National University, Busan, South Korea}
\email{yeongrak.kim@pusan.ac.kr}

\begin{document}

\begin{abstract}
    The Leray number of an abstract simplicial complex is the minimal integer $d$ where its induced subcomplexes have trivial homology groups in dimension $d$ or greater. We give an upper bound on the Leray number of a complex in terms of how the facets are attached to each other. We also describe the structure of complexes for the equality of the bound that we found.
    Through the Stanley--Reisner correspondence, our results give an Eisenbud--Goto type inequality for any square-free monomial ideals. This generalizes Terai's result.
\end{abstract}

\thanks{Corresponding author: Jinha Kim (\texttt{jinhakim@ibs.re.kr}).}



\maketitle

\section{Introduction}
Castelnuovo--Mumford regularity (simply called regularity) is one of the fundamental invariants in commutative algebra and algebraic geometry which measures the complexity of a graded module over a polynomial ring, or a coherent sheaf on a projective space. 
For a finitely generated graded module $M$ over a polynomial ring $S=k[x_1, \cdots, x_n]$ over an algebraically closed field $k$ of characteristic zero, we may read off the graded Betti numbers
$$ \beta_{i,i+j}(M) = \rank \operatorname{Tor}_i^S (M,k)_{i+j} $$
from the minimal free resolution of $M$.
Following an interpretation of Eisenbud and Goto \cite{MR741934}, the regularity of $M$ equals the height of the (Betti) table of graded Betti numbers.
Computing the regularity explicitly, or finding a nice upper bound in terms of algebraic and geometric invariants of $M$ is thus a fundamental question in commutative algebra and algebraic geometry. 
The most famous question, often called the regularity conjecture of Eisenbud and Goto \cite{MR741934}, asks whether
$$\reg(I_X) \le \deg (X) - \codim (X) +1 $$
where $I_X \subset S$ is the homogeneous prime ideal of a nondegenerate projective variety $X \subseteq \PP (S)$. 
McCullough and Peeva pointed out that the conjecture is not true in general, however, there are several important cases that fit into the conjecture. 
Indeed, the conjecture is still mysterious for smooth (or mildly singular) projective varieties. 
It seems to be that finding an upper bound for the regularity of the form $\deg - \codim + 1 + (\text{some constants})$ is also an interesting question \cite{MR2601633}. 
We refer to \cite{MR3758150} and references therein for more details on this story.

Given a variety and its defining ideal, one can naturally obtain a monomial ideal that preserves many algebraic invariants by taking its \textit{initial ideal} (with respect to any monomial order).
Moreover, thanks to Bayer and Stillman \cite{MR894583}, the regularity of an ideal $I \subset S$ equals the regularity of the \textit{generic initial ideal} of $I$ with respect to the reverse lexicographic order.
Together with the notion of polarization of a monomial ideal, the regularity of an ideal can be computed from the regularity of the square-free monomial ideal in $S$.

Elegant works by Stanley, Hochster, and Reisner in the 1970s bring us to the realm of combinatorics, especially in terms of combinatorial/homological data of the corresponding simplicial complex.
In detail, there is a one-to-one correspondence between square-free monomial ideals and simplicial complexes by associating monomial generators of the ideal with non-faces of the complex.
It is so-called the Stanley--Reisner correspondence.
Through this correspondence, we can read some numerical information about the ideal such as degree and codimension from the associated simplicial complex. 
Moreover, the graded Betti numbers of the monomial ideal can be computed combinatorially due to Hochster's formula.
Based on this correspondence, there are active studies about monomial ideals regarding the associated complexes, and the monomial ideals are called (non-)edge ideals in the literature.
(See \cite{MR3070118, MR3213523, MR2563591} for examples and \cite{MR2932582} for a survey.)

On the other hand, the regularity of square-free monomial ideals was independently studied as a homological dimension, called the ``Leray number'', of corresponding simplicial complexes. 
A simplicial complex $X$ is said to be {\em $d$-Leray} if for every $i \geq d$, the $i$-th homology group of any induced subcomplex of $X$ is trivial.
The {\em Leray number}, denoted by $L(X)$, is the smallest integer $d$ where $X$ is $d$-Leray.
See Section~\ref{subsec:complex} for terminologies about simplicial complexes.
The concept of the Leray number was introduced by \cite{Weg75} in the study of understanding the intersection patterns of convex sets based on Helly type theorems, and has been deeply studied and applied in various directions.
It follows from Hochster's formula that the Leray number of a simplicial complex and the regularity of the corresponding square-free monomial ideal are equivalent.
We refer to \cite{KM06} for more details about the correspondence.
See also \cite{Tancer} for an overview of Helly type theorems and $d$-Leray complexes.

Note that, since the Stanley--Reisner ideal $I_{X}$ corresponding to a simplicial complex $X$ needs not to be a prime ideal, there is no reason that the regularity of $I_X$ satisfies the Eisenbud--Goto inequality.
On the other hand, Terai showed that the Eisenbud--Goto inequality for $\reg (I_X)$ is valid under the assumption that the simplicial complex $X$ is pure and strongly connected in \cite{MR1838920}. 
Motivated by his work, we propose a question on an upper bound of the regularity of a Stanley--Reisner ideal (or equivalently, the Leray number of a simplicial complex) for general situations. 

In this paper, we introduce a new function $M(X)$ on a simplicial complex $X$ that provides an upper bound for the Leray number $L(X)$.
It is defined inductively on the number of facets and depending on how the facets of the complex are attached to each other.
See \Cref{thm:main} for details.
A key observation is that the Leray number of a simplicial complex increases by at most one whenever we attach a new simplex.

In addition, we explore some classes of simplicial complexes that satisfy the equality established in our main result.
More explicitly, in Theorem~\ref{thm:eqaulity}, we investigate \textit{weak shellable complexes} that satisfy the equality stated in \Cref{thm:main}.
As an application, in \Cref{cor:weakEG}, we discover an analogue of Terai's result by giving a \textit{weak Eisenbud--Goto inequality} for all simplicial complexes. 
Also, \Cref{cor:Weakeular} gives a bound on $M(X)$ in terms of the \textit{weighted Euler characteristic} of the weighted graph associated with $X$.

The structure of the paper is as follows. In Section~\ref{sec:prem}, we list preliminaries and backgrounds. We state and prove our main results in Section~\ref{Sec:main} and \ref{sec:main2}. In Section~\ref{sec:rmk}, we discuss some applications and remarks in the view of algebraic geometry and commutative algebra.

\section{preliminaries}\label{sec:prem}
We review some basic definitions and notions first.
For details, please consult some textbooks such as \cite{MR2724673}, \cite{MR2110098}, or \cite{MR1453579}.
\subsection{(abstract) simplicial complexes}\label{subsec:complex}
A \textit{simplicial complex} $X$ on $V$ is a collection of subsets of $V$ that is closed under taking subsets, that is, if $F \in X$ and $G \subset F$ then $G \in X$.
$V$ is called the {\em vertex set} of $X$, and the elements of $X$ are called the \textit{faces}.
The faces that are maximal with respect to inclusion are called the \textit{facets}.
We say $X$ is a \textit{simplex} on $V$ if $X=2^V$ (including the case when $V=\emptyset$).
If $X$ does not have any face, that is, if $X$ is indeed an empty set, then we often call it a {\em void complex}.

For a subset $W$ of $V$, the \textit{induced subcomplex} of $X$ on $W$, denoted by $X[W]$, is the simplicial complex on $W$ whose faces are the faces of $X$ that are contained in $W$.
For a face $F$ of $X$, we define the \textit{dimension} of $F$ by $\dim F :=|F|-1$ and the \textit{dimension} of $X$ by $\displaystyle{\dim X = \max_{F\in X} \dim F}$.
Let $f_i$ be the number of $i$-dimensional faces of $X$ for $i\ge -1$.
When $\dim X=d-1$, $f_0$ denotes the number of vertices and $f_{d-1}$ denotes the number of the facets with the maximum dimension of the complex.

For a vertex subset $\sigma$, we denote the simplex on $\sigma$ by $\Delta(\sigma)$. 
For a nonnegative integer $n$, we denote by $\Delta(n)$ a simplex on $n$ vertices.
Let $[n]:=\{1,\dots,n\}$ 
and denote by $\partial\Delta(n)$ the boundary of $\Delta(n)$, that is, if $\Delta(n)$ is a simplex on $[n]$ then $\partial\Delta(n)$ equals $\Delta(n)\setminus\{[n]\}$.

A simplicial complex $X$ is called \textit{pure} if all its facets have the same dimension.
When $X$ is pure and all facets have the dimension $d$, it is said to be \textit{strongly connected} if, for any two facets $\tau_1$ and $\tau_2$ of $X$, there exists a sequence of facets $(\tau_1=\sigma_1,\sigma_2,\dots,\sigma_{\ell}=\tau_2)$ for some $\ell \ge 1$ such that $\sigma_i \cap \sigma_{i+1}$ is a face of dimension $d-1$ for each $i=1,\dots,\ell-1$.

\subsection{Homology groups (over $\ZZ/2\ZZ$) of simplicial complexes}
Let $X$ be a non-void simplicial complex and $n \ge 0$ be an integer.
We define the \textit{$n$-th homology group} of $X$ as follows.
Here, we only consider homology groups with coefficients in $\mathbb{Z}_2 \cong \ZZ/2\ZZ$.

Let $C_n(X)$ be the vector space over $\mathbb{Z}_2$ where its basis is the $n$-dimensional faces of $X$. That is,
\[C_n(X)=\{\sum_{i=1}^k c_i \sigma_i : c_i \in \mathbb{Z}_2, \sigma_i \text{ is an $n$-dimensonal face of }X\}.\]
We define $C_{-1}(X)=\mathbb{Z}_2$.
Each element of $C_n(X)$ is called an \textit{$n$-chain} of $X$.

For an integer $n \ge 1$, the \textit{$n$-th boundary map} $\partial_n$ is the linear map from $C_n(X)$ to $C_{n-1}(X)$ where $\partial_n(\sigma)=\sum_{i=1}^{n+1}(\sigma\setminus\{v_i\})$ for an $n$-dimensional face $\sigma=\{v_1,\dots,v_{n+1}\}$ of $X$.
For $n=0$, $\partial_0: C_0(X) \to C_{-1}(X)$ is defined as the linear map where $\partial_0(v)=1$ for each vertex $v$ of $X$.

Now we have a sequence of linear maps of vector spaces over $\mathbb{Z}_2$
\[\cdots \to C_{n+1}(X) \xrightarrow{\partial_{n+1}} C_n(X) \xrightarrow{\partial_n} C_{n-1}(X) \to \cdots \to C_0(X) \xrightarrow{\partial_{0}} C_{-1}(X) \to 0\]
with $\partial_n \partial_{n+1}=0$ for each $n$.
This sequence is called the \textit{chain complex} of $X$.
Since $\partial_{n} \partial_{n+1}=0$, we have $\im \partial_{n+1} \subset \ker \partial_{n}$.
Then the \textit{$n$-th (reduced) homology group} of $X$ is defined as $\tilde{H}_n(X)=\ker \partial_n / \im \partial_{n+1}$.
Each element of $\ker\partial_n$ is called an \textit{$n$-cycle} of $X$ and each element of $\im \partial_{n+1}$ is called an \textit{$n$-boundary} of $X$.
For an $n$-cycle $\alpha$, we have $\alpha+\im\partial_{n+1} \in \tilde{H}_n(X)$ and denote it by $[\alpha]$ for convenience.
The \textit{$n$-th (reduced) Betti number} of $X$ is $\tilde{\beta}_n(X)=\dim \tilde{H}_n(X)$.
Recall that a simplicial complex $X$ is \textit{$d$-Leray} if $\tilde{H}_i(X[W])=0$ for all $i \ge d$ and all $W \subset V$. 

The Mayer--Vietoris exact sequence is one of the most fundamental tools that help us understand the homology groups. 
\begin{theorem}\label{thm:mvseq}
    For simplicial complexes $A$ and $B$, we have the following exact sequence:
        \begin{align*}
        \begin{split}
        & \cdots \longrightarrow \tilde{H}_n(A \cap B) \quad  \xrightarrow{(i_n^*, j_n^*)} \quad  \tilde{H}_n(A) \oplus \tilde{H}_n(B) ~ \longrightarrow ~\tilde{H}_n(A \cup B)\\
        & \xrightarrow{~ \delta_n ~} \tilde{H}_{n-1}(A \cap B) \xrightarrow{(i_{n-1}^*,j_{n-1}^*)} \tilde{H}_{n-1}(A) \oplus \tilde{H}_{n-1}(B) \longrightarrow \tilde{H}_{n-1}(A \cup B) \longrightarrow \cdots.
        \end{split}
        \end{align*}
        Here, $i_n^*:\tilde{H}_n(A \cap B) \to \tilde{H}_n(A)$ and $j_n^*:\tilde{H}_n(A \cap B) \to \tilde{H}_n(B)$ are the induced maps from the inclusion maps $i: A\cap B \hookrightarrow A$, $j: A\cap B \hookrightarrow B$, respectively.
        Furthermore, the map $\delta_n$ is defined as follows:
        For each $n$-cycle $\alpha$ of $X$, if $\alpha=\alpha_1+\alpha_2$ for an $n$-cycle $\alpha_1$ of $A$ and an $n$-cycle $\alpha_2$ of $B$, then $\delta_n([\alpha])=[\partial_n(\alpha_1)]=[\partial_n(\alpha_2)]$.
\end{theorem}

\subsection{Stanley--Reisner correspondence}
Let $I$ be an ideal generated by square-free monomials in the polynomial ring $S = k[x_1,\dots,x_n]$ over a field $k$. We can associate with $I$ a simplicial complex, denoted by $X_I$, whose faces correspond to monomials that are not elements of $I$. That is, $X_I$ consists of all subsets $F\subseteq [n]$ such that $\prod_{i\in F} x_i \notin I$. The simplicial complex $X_I$ is called the \textit{Stanley--Reisner complex} of $I$.

Conversely, given a simplicial complex $X \subseteq 2^{[n]}$, we can associate with it a square-free monomial ideal $I_X$, generated by the non-faces of the complex $X$. That is, $I_X$ is the ideal in $S$ generated by all monomials $\prod_{i\in F} x_i$, where $F\notin X$. The monomial ideal $I_X$ is called the \textit{Stanley--Reisner ideal} of $X$.
This correspondence between square-free monomial ideals and simplicial complexes is known as the Stanley--Reisner correspondence.

One important aspect of the Stanley--Reisner correspondence is that the graded Betti numbers of the Stanley--Reisner ideal can be computed as the homologies of induced subcomplexes of the corresponding simplicial complex.
In this paper, we introduce a homological version of Hochster's formula.
\begin{theorem}\cite{MR0441987}\label{thm:hochster}
    Let $I$ be a Stanley--Reisner ideal in the polynomial ring $S = k[x_1,\dots,x_n]$, where $k$ is a field of characteristic zero. 
    Let $X_I$ be the Stanley--Reisner complex of $I$.
    Then, for $j\ge 2$, $$\beta_{i,i+j}(I) = \sum_{W\subset V} \dim_k \tilde{H}_{j-2}(X_I[W])$$
    where $W$ runs over $V$ with $|W|=i+j$.
\end{theorem}

We must work over a field of characteristic zero in this version of the formula because it is a necessary condition for identifying the rank of cohomologies of the upper Koszul complex with the dimension of homologies of the Stanley--Reisner complex. It is discussed in \cite[Remark 7.15]{MR3213521}.
By identifying the complex and the monomial ideal, we may abuse the notation: $\beta_{i,i+j}(I)=\beta_{i,i+j}(X_I)$.

Therefore, regarding the definition of topological Betti numbers, one can view the topological Betti numbers $\tilde{\beta}_{j-2}(X)$ as the graded Betti numbers $\beta_{i,i+j}(X)$ where $i+j = |V|$, for the vertex set $V$ of $X$.

Meanwhile, the algebraic sets defined by the Stanley--Reisner ideals correspond to the union of incomparable linear subspaces of a projective space. This union of linear subspaces in a projective space is known as a (coordinate) linear subspace arrangement in the projective space. 
Let $Y$ be a linear subspace arrangement in $\PP^N$, i.e., $Y= \bigcup_{i=1}^{m} L_i$ where $L_i$ are incomparable linear subspaces in $\PP^N$. 
In case the base field of the linear subspaces is an infinite field due to Hilbert's Nullstellensatz, there is a one-to-one correspondence between the linear subspaces $L_1,\ldots,L_m$ in the linear subspace arrangement and the facets $\sigma_1,\ldots,\sigma_m$ of the Stanley--Reisner complex $X_I$.
Note that this correspondence would not be guaranteed if the base field is a finite field. (See \cite[Remark 1.9]{MR2110098}.)

\subsection{Some correspondences between invariants}
Through the Stanley--Reisner correspondence, there are some correspondences between algebraic invariants on monomial ideals and topological invariants on simplicial complexes.

Let $I_X$ be the Stanley--Reisner ideal of a simplicial complex $X$ on $n$ vertices with $\dim X = d - 1$.
Using the Stanley--Reisner correspondence, we can relate algebraic invariants on $I_X$ to topological invariants on $X$. 
In particular, we can employ the correspondence between the $h$-vector and the $f$-vector to establish that the degree of the face ring $k[X]$ equals $f_{d-1}(X)$.
In addition, the (Krull) dimension of the face ring $k[X]$ equals $\dim X=d-1$. Hence, we have $\codim(k[X])=n-d$.

The Castelnuovo--Mumford regularity is an algebraic invariant that captures the complexity of homogeneous ideals. 
It measures the highest degree of entries in the differentials of the graded minimal free resolution of an ideal $I$. 
\begin{definition}[cf. {\cite[p. 48]{MR2724673}} or {\cite[Definition 5.54]{MR2110098}}]
$$\reg(I) = \max \{j: \beta_{i,i+j}(I)\not= 0 \text{ for some } i\}$$
\end{definition}
In other words, the regularity of $I$ is the \textit{height} of the Betti diagram, which records the graded components of the minimal free resolution of $I$.

If an ideal $I$ is a Stanley-Reisner ideal, the regularity of $I$ can be obtained from the vanishing homological degree of the induced subcomplexes of the Stanley-Reisner complex $X_I$ using \Cref{thm:hochster}.
More explicitly, we have $$\reg(I)=\max\{j:\tilde{H}_{j-2}(X_I[W]) \not= 0\text{ for some }W\subseteq V(X_I)\},$$ which implies that $L(X_I) = \reg(I) - 1$.

\begin{example}\label{eg:SRcorr}
    Let $I = (x_1x_2x_3,\, x_3x_4,\, x_2x_5,\, x_1x_4x_5,\, x_1x_6 ,\, x_2x_4x_6,\, x_3x_5x_6)$ in a polynomial ring $k[x_1,\dots,x_6]$.
    The Stanley--Reisner complex $X$ of $I$ is $\Delta(\{1,2,4\}) \cup \Delta(\{1,3,5\}) \cup \Delta(\{2,3,6\}) \cup \Delta(\{4,5,6\})$, and the coordinate linear subspace arrangement defined by $I$ is the union of $4$ planes in $\PP^5$.

    \begin{figure}[h!]
        \centering
        \begin{tabular}{r c r c c}
        & & & & \multirow{6}{10em}{\includegraphics[scale=0.5]{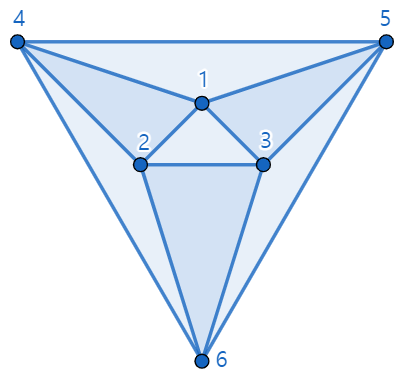}} \tabularnewline
        $I=(x_1x_2x_3,\; x_3x_4,$ & $\leftrightarrow$ & $X=\Delta(\{1,2,4\})$ & & \tabularnewline
        $x_2x_5,\; x_1x_4x_5,$  & & $\cup \Delta(\{1,3,5\})$ & & \tabularnewline
        $x_1x_6,\; x_2x_4x_6,$ & & $\cup \Delta(\{2,3,6\})$ & &  \tabularnewline
        $x_3x_5x_6)$ & & $\cup \Delta(\{4,5,6\})$ & & \tabularnewline
        & & & & \tabularnewline
        & & & & \tabularnewline
        \end{tabular}
        \caption{The Stanley--Reisner correspondence}
        \label{fig:Invcorrespondences}
    \end{figure}
    Since the simplicial complex $X$ consists of $4$ facets of dimension $2$, the face ring $k[X]$ has degree $4$, and the codimension of $k[X]$ is $3$.
    The regularity of $I$ is $3$, and the Leray number of $X$ is $2$.
\end{example}

\section{Bounds on Leray numbers of simplicial complexes}\label{Sec:main}

Let $X$ be a simplicial complex with $m$ facets.
Take a linear order $\prec:\sigma_1 \prec \cdots \prec\sigma_m$ of the facets of $X$.
For $j \in [m]$, let $X_j=\bigcup_{i=1}^{j} \Delta(\sigma_i)$.
We define $M_{\prec}(X)$ inductively as follows:
$M_{\prec}(X_1) = 1$ and  for $j = 2,\dots,m$,
\[M_{\prec}(X_j) = \begin{cases}
M_{\prec}(X_{j-1}) & \text{if } X_{j-1}\cap \Delta(\sigma_j) \text{ is a simplex,}\\
M_{\prec}(X_{j-1})+1 & \text{otherwise.}
\end{cases}\]
We define $M(X)=min\{M_{\prec}(X) :\; \prec \text{ is a linear order of the facets of }X\}$.

Now we prove that $M(X)$ is an upper bound for the Leray number $L(X)$ of $X$.
Our proof is based on the induction on the number of the facets of $X$ and the following lemma will be used for the inductive step.

\begin{lemma}\label{lem:inductivestep}
    Let $X$ be a simplicial complex on $V$ and let $\prec: \sigma_1 \prec \cdots \prec \sigma_m$ be a linear ordering of the facets of $X$ such that $M(X)=M_{\prec}(X)$. Let $X' =\cup_{i=1}^{m-1}\Delta(\sigma_i)$ and $\prec'$ be the induced order of $\prec$ on $\{\sigma_1,\ldots,\sigma_{m-1}\}$, which is a linear order of the facets of $X'$.
    For a subset $W$ of $V$, we have the following long exact sequence:
    \begin{align}\label{mvseq}
    \begin{split}
    & \cdots \to \tilde{H}_i(X'[W]\cap \Delta(\sigma_m)[W]) \xrightarrow{\iota_i^*} \tilde{H}_i(X'[W]) \to \tilde{H}_i(X[W])\\
    & \xrightarrow{\delta_i}  \tilde{H}_{i-1}(X'[W]\cap \Delta(\sigma_m)[W]) \to \tilde{H}_{i-1}(X'[W]) \to \tilde{H}_{i-1}(X[W]) \to \cdots.
    \end{split}
    \end{align}
    Furthermore, the following hold.
    \begin{enumerate}
        \item[(i)] If $X' \cap \Delta(\sigma_m)$ is a simplex, then 
        \[L(X)=L(X') \text{ and } M(X)=M_{\prec}(X)=M_{\prec'}(X')\ge M(X').\]
        \item[(ii)] If $X' \cap \Delta(\sigma_m)$ is not a simplex, then 
        \[L(X) \leq L(X')+1 \text{ and } M(X)=M_{\prec}(X)=M_{\prec'}(X')+1 \ge M(X')+1.\]
    \end{enumerate}
\end{lemma}
\begin{proof}
    First, note that we can consider both $X'$ and $\Delta(\sigma_m)$ as simplicial complexes on $V$.
    Take a subset $W$ of $V$.
    By applying Theorem~\ref{thm:mvseq} to $X[W]=X'[W] \cup \Delta(\sigma_m)[W]$, we obtain the long exact sequence~\eqref{mvseq}.
    Note that $\Delta(\sigma_m)[W]$ is a simplex, and hence it has the trivial homology group for every dimension.

    To prove (i), assume that $X' \cap \Delta(\sigma_m)$ is a simplex.
    Then $X'[W]\cap \Delta(\sigma_m)[W]$ is also a simplex, and hence it has the trivial homology group for every dimension.
    Thus, by \eqref{mvseq}, we have $\tilde{H}_i(X'[W]) \simeq \tilde{H}_i(X[W])$ for every $i$.
    This implies that $L(X)=L(X')$.
    In addition, by the definition of $M_{\prec}(X)$, we have $M_{\prec}(X)=M_{\prec'}(X')$.
    Thus, we obtain $$M(X)=M_{\prec}(X)=M_{\prec'}(X')\ge M(X').$$

    Now, to prove (ii), assume that $X' \cap \Delta(\sigma_m)$ is not a simplex.
    Note that $X'[W]\cap \Delta(\sigma_m)[W]$ is an induced subcomplex of $X'$ since $\Delta(\sigma_m)[W]$ is a simplex.
    Then we have $\tilde{H}_i(X'[W])=\tilde{H}_i(X'[W]\cap \Delta(\sigma_m)[W])=0$ for $i \ge L(X')$.
    Thus, by \eqref{mvseq}, we obtain that $\tilde{H}_i(X[W])=0$ for every $i \ge L(X')+1$.
    This proves that $L(X) \le L(X')+1$.
    Since we have $M_{\prec}(X)=M_{\prec'}(X')+1$ by the definition of $M_{\prec}(X)$, we obtain $$M(X)=M_{\prec}(X)=M_{\prec'}(X')+1 \ge M(X')+1$$ as desired.
    \end{proof}

\begin{theorem}\label{thm:main}
For every simplicial complex $X$, we have $L(X)\le M(X)$.
\end{theorem}
\begin{proof}
We induct on the number $m$ of facets of the complex $X$.
Take a linear order $\prec:\sigma_1 \prec \cdots \prec\sigma_m$ such that $M(X)=M_{\prec}(X)$.
If $m = 1$, then $X=\Delta(\sigma_1)$ is a simplex.
In this case, we have $L(X) = 0 \le M(X)=1$.

Now, suppose $m \ge 2$.
Let $X' = \cup_{i=1}^{m-1} \Delta (\sigma_i)$ and assume that $L(X')\le M(X')$.
Let $\prec'$ be the induced order of $\prec$ on the facets of $X'$.
We divide it into two cases as follows.
\begin{enumerate}
    \item[(i)] $X'\cap \Delta(\sigma_m)$ is a simplex.
    \item[(ii)] $X' \cap \Delta(\sigma_m)$ is not a simplex.
\end{enumerate}
In Case (i), by Lemma~\ref{lem:inductivestep}, we have 
\[L(X)=L(X') \text{ and } M(X)=M_{\prec}(X)=M_{\prec'}(X')\ge M(X').\]
Thus, we conclude that 
\[L(X)=L(X') \le M(X') \le M_{\prec'}(X')=M_{\prec}(X)=M(X).\]
In Case (ii), by Lemma~\ref{lem:inductivestep}, we have
\[L(X) \leq L(X')+1 \text{ and } M(X)=M_{\prec}(X)=M_{\prec'}(X')+1 \ge M(X')+1.\]
Hence, we also obtain that 
\[L(X) \leq L(X')+1 \leq M(X')+1 \leq M_{\prec'}(X')+1 =M_{\prec}(X)=M(X).\]
Therefore, we have $L(X) \le M(X)$ in both cases.
\end{proof}

\section{The equality cases}\label{sec:main2}
Let $X$ be a simplicial complex and $\prec:\sigma_1 \prec \cdots \prec \sigma_m$ be a linear order of facets of $X$.
Let $X_j=\cup_{i=1}^{j}\Delta(\sigma_i)$.
We say $\prec$ is a \textit{weak shelling} of $X$ if 
there is a vertex $u_j$ such that $X_{j-1} [\sigma_j\setminus\{u_j\}]$ is a simplex for each $j=2,\ldots,m$.
Note that this definition coincides with a weak shelling that was discussed in \cite[Section 5]{MR4603825} for pure complexes, since the edge set of a hypergraph can be viewed as the set of all facets of the corresponding simplicial complex.

The purpose of this section is to characterize the equality case of Theorem~\ref{thm:main} where a simplicial complex $X$ admits a weak shelling $\prec$ such that $M_{\prec}(X)=M(X)$.

\begin{theorem}\label{thm:eqaulity}
    Let $X$ be a non-empty simplicial complex on $V$ with $L(X)=M(X)$. Suppose there is a weak shelling $\prec$ of $X$ such that $M_{\prec}(X)=M(X)$. Then each of the following holds:
    \begin{enumerate}
        \item $L(X)=1$ $\iff$ $M(X)=1$ and $X$ is not a simplex.
        \item $L(X)=2$ $\iff$ $M(X)=2$ and $\tilde{H}_1(X[W]) \neq 0$ for some $W \subset V$. If $\tilde{H}_1(X[W]) \neq 0$, then $X[W]$ contains an induced cycle, which corresponds to a generator of $\tilde{H}_1(X[W])$.
        \item $L(X)=k$ for $k \ge 3$ $\iff$ $M(X)=k$ and $\tilde{H}_{k-1}(X[W]) \neq 0$ for some $W \subset V$.
        If $\tilde{H}_{k-1}(X[W]) \neq 0$, then $X[W]$ contains an induced subcomplex isomorphic to $\partial \Delta(k+1)$, which corresponds to a generator of $\tilde{H}_{k-1}(X[W])$.
    \end{enumerate}
    Furthermore, if $L(X)=k$ for $k \ge 2$, then $\tilde{\beta}_{k-1}(X[W]) \le 1$ for every $W \subset V$.
\end{theorem}

\begin{remark}
    Theorem~\ref{thm:eqaulity} may not hold for a simplicial complex $X$ that does not admit a weak shelling $\prec$ such that $M_\prec(X) = M(X)$. For example, let $X$ be a complex on $\{1,2,3,4,5\}$ with facets $\{1234, 125, 235, 345, 145\}$. Here, $M(X)=3=L(X)$ but clearly $X$ does not contain the boundary of a simplex on $4$ vertices.
\end{remark}

We will present a proof of Theorem~\ref{thm:eqaulity} in the end of this section.

\begin{lemma}\label{lem:inducedorder}
    Let $X$ be a simplicial complex on $V$ and $\prec$ be a weak shelling of $X$.
    Then, for every $W \subset V$, there exists a weak shelling $\prec_W$ of $X[W]$ such that $M_{\prec_W}(X[W]) \le M_{\prec}(X)$.
\end{lemma}
\begin{proof}
    Suppose $\prec$: $\sigma_1 \prec \cdots \prec \sigma_m$ is a weak shelling of $X$ and let $W$ be any subset of~$V$.
    Then each facet of $X[W]$ can be expressed as the intersection $\sigma_i \cap W$ for some $i\in[m]$.
    Let $$U = \{j\in[m]: \text{there is no }i<j\text{ such that }\sigma_i \cap W=\sigma_j \cap W\}.$$
    Then $\{\sigma_j \cap W: j\in U\}$ is precisely the set of all facets of $X[W]$.
    Define a linear order $\prec_W$ of facets of $X[W]$ by $\sigma_i \cap W \prec_W \sigma_j \cap W$ if $\sigma_i \cap W \neq \sigma_j \cap W$ and $i<j$.
    It immediately follows from the definition that $\prec_W$ is a weak shelling of $X[W]$ and $M_{\prec_W}(X[W]) \le M_{\prec}(X)$.
\end{proof}

The following lemma is a key ingredient that will be used in an inductive argument of the proof of Theorem~\ref{thm:eqaulity}. 

\begin{lemma}\label{lem:commonvertex}
Let $k\geq3$ be an integer, and let $Y$ be a simplicial complex with a weak shelling $\prec$ such that $L(Y)=M_{\prec}(Y)=k-1$.
Suppose $\{v_1,\dots,v_k\}$ induces a subcomplex of $Y$ isomorphic to $\partial\Delta(k)$, and for each $i$ let $g_i=\{v_1,\dots,v_k\}\setminus\{v_i\}$.
If there are $(k-1)$-dimensional faces $\tau_1,\ldots,\tau_q$ of $Y$ such that $\partial_{k-1}(\sum_{i=1}^{q}\tau_i)=\sum_{i=1}^{k}g_i$, then there is a vertex $w \in \cup_{i=1}^{q}\tau_i$ such that $g_i\cup\{w\} \in Y$ for each $i \in [k]$.
\end{lemma}
\begin{proof}
    Let $\prec: \sigma_1 \prec \cdots \prec \sigma_s$ be a weak shelling of $Y$ such that $M_{\prec}(Y)=k-1$. Let $Y_j = \bigcup_{i=1}^{j} \Delta(\sigma_i)$ and $\tau = \sum_{i=1}^{q}\tau_i$.

    We first define a function $n:[k]\to[s]$ such that for each $j\in[k]$, $n(j)$ is the smallest integer $\ell$ where $\sigma_{\ell}$ contains $g_j$.
    Note that $n(j) \neq n(j')$ if $j \neq j'$: whenever $j \neq j'$, the union $g_j\cup g_{j'}=\{v_1,\dots,v_k\}$ is not a face of $Y$.
    By rearranging $g_i$'s if necessary, we may assume that $n(1)<\dots<n(k)$.

    We claim that $Y_{n(j)-1} \cap \Delta(\sigma_{n(j)})$ is not a simplex for each $3 \le j \le k$, and hence $\sigma_{n(3)}, \ldots, \sigma_{n(k)}$ are precisely the $k-2$ facets of $Y$ that make $M_{\prec}(Y) = k-1$.
    First observe that $$V(Y_{n(j)-1} \cap \Delta(\sigma_{n(j)})) \supset (\sigma_{n(1)}\cup\sigma_{n(2)})\cap\sigma_{n(j)} \supset g_j.$$
    Thus, if $Y_{n(j)-1} \cap \Delta(\sigma_{n(j)})$ is a simplex, there must be a facet $\sigma_{j'}$ of $Y$ with $j' < n(j)$ such that $g_j \subset \sigma_{j'}$, which is a contradiction to the minimality of $n(j)$.
    Therefore, $Y_{n(j)-1} \cap \Delta(\sigma_{n(j)})$ is not a simplex.

    Let $Z = Y_{n(k)-1}$, and let $\prec'$ be the induced order of $\prec$ on $\{\sigma_1,\ldots,\sigma_{n(k)-1}\}$. By the above observation, it is clear that $M_{\prec'}(Z) = k-2$.
    It follows from Theorem~\ref{thm:main} that $L(Z) \leq k-2$.
    On the other hand, noting that $Z[g_k] \cong \partial\Delta(k-1)$, we have $L(Z) \geq k-2$.
    Combining the two inequalities, we obtain $L(Z) = M_{\prec'}(Z) = k-2$.

    Next, let $m$ be the smallest integer such that $Y_m$ contains the $(k-1)$-chain $\tau$. It is obvious that $m \geq n(k)$. 
    We claim that $Y_{n(k)}$ has a $(k-1)$-chain $\lambda = \sum_{i=1}^{q'}\lambda_i$ such that $\bigcup_{i=1}^{q'}\lambda_i \subset \bigcup_{i=1}^{q}\tau_i$ and $\partial_{k-1}(\sum_{i=1}^{q}\lambda_i)=\sum_{i=1}^{k}g_i$.
    The claim is obviously true when $m = n(k)$, so we may assume $m > n(k)$.
    Clearly, $Y_{m-1}\cap\Delta(\sigma_m)$ must be a simplex: otherwise, it is a contradiction to $M_\prec(Y) = k-1$. Since $Y_m = Y_{m-1} \cup \Delta(\sigma_m)$, $\tau$ can be expressed as $\tau=\kappa_1+\kappa_2$ where $\kappa_1$ is a $(k-1)$-chain of $Y_{m-1}$ and $\kappa_2$ is a $(k-1)$-chain of $\Delta(\sigma_m)$. Since $$m > n(k)\;\;\text{and}\;\;\partial_{k-1}(\kappa_2) = \partial_{k-1}(\tau) - \partial_{k-1}(\kappa_1) = \sum_{i=1}^k g_i - \partial_{k-1}(\kappa_1),$$ we observe that $\partial_{k-1}(\kappa_2)$ is a $(k-2)$-cycle of $Y_{m-1}\cap\Delta(\sigma_m)$.
    Since $Y_{m-1}\cap\Delta(\sigma_m)$ is a simplex, we can find a $(k-1)$-chain $\kappa_2'$ of $Y_{m-1}$ such that the vertex set of $\kappa_2'$ equals that of $\partial_{k-1}(\kappa_2)$ and $\partial_{k-1}(\kappa_2')=\partial_{k-1}(\kappa_2)$. 
    Then $\kappa_1+\kappa_2'$ is a $(k-1)$-chain of $Y_{m-1}$ where $\partial_{k-1}(\kappa_1+\kappa_2') = \sum_{i=1}^k g_i$ and all vertices of $\kappa_1+\kappa_2'$ belong to the vertex set of $\tau$. 
    By repeating the above argument, we can find a $(k-1)$-chain $\lambda$ of $Y_{n(k)}$ where $\partial_{k-1}(\lambda) = \sum_{i=1}^k g_i$ and all vertices of $\lambda$ belong to the vertex set of $\tau$, as desired.

    Now, let  $\lambda = \sum_{i=1}^{q'}\lambda_i$.
    We may assume that $Z \cap \{\lambda_1,\ldots,\lambda_{q'}\} = \{\lambda_1,\ldots,\lambda_p\}$ for some $1 \leq p \leq q'$.
    In particular, $\lambda_{p+1},\ldots,\lambda_{q'} \subset \sigma_{n(k)}$.
    Let $e_1,\dots,e_{\ell}$ be the $(k-2)$-dimensional faces that have non-zero coefficients in $\sum_{i=p+1}^{q'}\partial_{k-1}(\lambda_i)$, that is, $\sum_{i=p+1}^{q'}\partial_{k-1}(\lambda_i)=\sum_{i=1}^{\ell}e_i$.
    From the observation that
    $$\sum_{i=1}^{p} \partial_{k-1}(\lambda_i)=\sum_{i=1}^{q'} \partial_{k-1}(\lambda_i)-\sum_{i=p+1}^{q'}\partial_{k-1}(\lambda_i)=\sum_{i=1}^{k}g_i-\sum_{i=p+1}^{q'}\partial_{k-1}(\lambda_i)$$
    and that $g_k$ is not a face of $Z$, we conclude that $g_k$ has a non-zero coefficient in $\sum_{i=p+1}^{q'}\partial_{k-1}(\lambda_i)$. That is, $g_k \in \{e_1,\ldots,e_\ell\}$. We may assume that $g_k = e_1$.
    Since $\partial_{k-2}(\sum_{i=1}^{\ell}e_i)=\partial_{k-2}(\sum_{i=p+1}^{q} \partial_{k-1}(\tau_i))=0$, we have $\partial_{k-2} (g_k)= \partial_{k-2}(e_1) =\partial_{k-2}(\sum_{i=2}^{\ell}e_i)$. 

    Since $\sum_{i=1}^{p}\partial_{k-1}(\lambda_i)=\sum_{i=1}^{k}g_i-\sum_{i=1}^{\ell}e_i = \sum_{i=1}^{k-1}g_i-\sum_{i=2}^{\ell}e_i$, and both $\sum_{i=1}^{p}\partial_{k-1}(\lambda_i)$ and $\sum_{i=1}^{k-1}g_i$ are $(k-2)$-chains of $Z$, $\sum_{i=2}^{\ell}e_i$ is also a $(k-2)$-chain of $Z$.
    Therefore, $\sum_{i=2}^{\ell}e_i$ is a $(k-2)$-chain of $Z$ such that its boundary is equal to $\partial_{k-2}(g_k)$.
    
    We next prove the following claims which is the last ingredient to apply the induction hypothesis to show the main statement.

\begin{claim}\label{claim:induc-step}
    Let $w\in \cup_{i=2}^{\ell}e_i$ be a vertex such that $(g_k \cup \{w\})\setminus\{v_i\}$ is a face of $Z$ for each $i \in [k-1]$. Then $g_i \cup \{w\} \in Y$ for each $i \in [k]$.
\end{claim}
\begin{proof}[Proof of Claim~\ref{claim:induc-step}]
    Since $w \in \cup_{i=2}^{\ell}e_i \subset \cup_{i=p+1}^{q'}\lambda_i$, we have $g_k\cup\{w\} \subset \sigma_{n(k)}$, so $g_k\cup\{w\} \in Y$.
    
    Now we show $g_i\cup\{w\} \in Z$ for each $i \in [k-1]$.
    Then we obtain that $g_i\cup\{w\} \in Y$ for each $i \in [k]$ as we wanted. 
    Let $\tilde{g_i}=g_k\cup\{w\}\setminus\{v_i\}$ and $\eta=\sum_{i=1}^{k-1}g_i+\sum_{i=1}^{k-1}\tilde{g_i}$.
    Note that $\partial_{k-2}(\sum_{i=1}^{k-1}g_i)=\partial_{k-2}(\sum_{i=1}^{k-1}\tilde{g_i})=\partial_{k-2}(g_k)$ since 
    both $(\sum_{i=1}^{k-1}g_i)+g_k$ and $(\sum_{i=1}^{k-1}\tilde{g_i})+g_k$ are isomorphic to $\partial\Delta(k)$.
    Then we have $\partial_{k-2}(\eta)=0$.
    Thus $\eta$ is a $(k-2)$-cycle of $Z[\{v_1,\dots,v_k,w\}]$.
    Since $L(Z) \le k-2$, there is a $(k-1)$-chain $\mu$ of $Z[\{v_1,\dots,v_k,w\}]$ such that $\partial_{k-1}(\mu)=\eta$.
    Since $g_k=\{v_1,\dots,v_{k-1}\}$ is not a face of $Z$, every $(k-1)$-dimensional face of $Z[\{v_1,\dots,v_k,w\}]$ is of the form $g_i\cup\{w\}$ for some $i \in [k-1]$.
    In order to have a $(k-1)$-chain $\mu$ such that $\partial_{k-1}(\mu)=\eta$, $g_i\cup\{w\}$ must be a face of $Z$ for every $i \in [k-1]$.
\renewcommand{\qedsymbol}{$\blacksquare$}
\end{proof}

    Now we prove that there is a vertex $w \in \bigcup_{i=1}^q \tau_i$ such that $g_i \cup \{w\} \in Y$ for each $i \in [k]$.
    We proceed by induction on $k$.
    Assume $k=3$.
    We first show that there is a vertex $w \in \cup_{i=2}^{\ell} e_i$ such that $\{v_1,w\}$ and $\{v_2,w\}$ are faces of $Z$.
    Since $Z$ has a $1$-chain $\sum_{i=2}^{\ell}e_i$ such that its boundary is equal to $\partial_{1}(g_3)=\partial_1(\{v_1,v_2\})$, the set of edges $\{e_2,\dots,e_{\ell}\}$ contains a path connecting $v_1$ and $v_2$.
    Thus we can take $\tilde{e}_1,\tilde{e}_2 \in \{e_2,\dots,e_{\ell}\}$ such that $\tilde{e_1}=\{v_1,w_1\}$ and $\tilde{e_2}=\{v_2,w_2\}$ for some $w_1,w_2 \in V(Z)\setminus\{v_1,v_2\}$.
    Note that $w_1$ and $w_2$ can be equal.
    Then $v_1,v_2,w_1,w_2 \in V(Z) \cap \sigma_{n(3)}$.
    Since $\prec$ is a weak shelling, there is a vertex $u \in V(Z) \cap \sigma_{n(3)}$ such that $(V(Z) \cap \sigma_{n(3)})\setminus\{u\}$ is a face of $Z$.
    Since $\{v_1,v_2\}$ is not a face of $Z$, $u$ is either $v_1$ or $v_2$.
    We may assume $u=v_1$, then we have $\{v_1,w_1\},\{v_2,w_1\} \in Z$.
    Here, $w_1$ is a vertex that we want.
    Note that $w_1 \in \cup_{i=2}^{\ell} e_i \subset \cup_{i=p+1}^{q'}\lambda_i \subset \cup_{i=1}^{q}\tau_i$.
    Since we have $\{v_1,w_1\}=(g_3\cup\{w_1\})\setminus\{v_2\}$ and $\{v_2,w_1\}=(g_3\cup\{w_1\})\setminus\{v_1\}$, Claim~\ref{claim:induc-step} implies that $g_i \cup \{w_1\} \in Y$ for each $i \in [3]$. 
    This proves the case when $k=3$.
    
    Now suppose $k \ge 4$. 
    Note that $L(Z)=k-2$ since $Z[g_k]$ is isomorphic to $\partial\Delta(k-1)$ and $\prec'$ is a weak shelling of $Z$ such that $M_{\prec'}(Z)=k-2$.
    Recall that $\partial_{k-2} (g_k)= \partial_{k-2}(e_1) =\partial_{k-2}(\sum_{i=2}^{\ell}e_i)$.
    By induction hypothesis, there is a vertex $w \in \cup_{i=2}^{\ell}e_i$ such that $g_k\cup\{w\}\setminus\{v_i\}$ is a face of $Z$ for each $i \in [k-1]$.
    Then by Claim~\ref{claim:induc-step}, $g_i \cup \{w\} \in Y$ for each $i \in [k]$.
    Since $w \in \cup_{i=2}^{\ell}e_i \subset \cup_{i=p+1}^{q'}\lambda_i \subset \cup_{i=1}^{q}\tau_i$, this completes the proof.
\end{proof}

\begin{proof}[Proof of Theorem~\ref{thm:eqaulity}]
    By Theorem~\ref{thm:main}, it is straightforward to see that the if part holds for each case.
    We will present proofs for the only if parts.
    
    Let $\sigma_1,\ldots,\sigma_m$ be the facets of $X$.
    Take a weak shelling $\prec:\sigma_1\prec\cdots\prec\sigma_m$ of $X$ such that $M_{\prec}(X)=M(X)$.
    First, note that $X$ is a simplex if and only if $L(X)=0$.
    If $X$ is a simplex, then $\tilde{H}_i(Y)=0$ for any induced subcomplex $Y$ and for any integer $i \ge 0$. Thus this implies that $L(X)=0$.
    If $X$ is not a simplex, then there are two vertices $v_1,v_2$ of $X$ such that $\{v_1,v_2\}$ is not a face of $X$. Then the induced subcomplex $X[\{v_1,v_2\}]$ has a non-vanishing homology group in dimension $0$, and hence $L(X) >0$.
    Thus we can observe that the assumption $L(X)=M(X)\ge 1$ implies that $X$ is not a simplex, i.e. $m \ge 2$.
    
    We proceed by induction on $m+L(X)$.
    If $L(X)=1$, it is trivial.
    If $m=2$, then the intersection of the two facets of $X$ must be a simplex, thus $M_{\ell}(X)=1$ for any linear order $\ell$ of the facets of $X$, implying $L(X)=M(X)=1$. 
    Now we may assume $m \ge 3$ and $L(X)=k \ge 2$.
    
    Let $X'=\cup_{i=1}^{m-1} \Delta(\sigma_i)$ and $\prec'$ be the induced order of $\prec$ on the facets of $X'$.
    We divide it into two cases.
    \begin{enumerate}
        \item[(i)] $X' \cap \Delta(\sigma_m)$ is a simplex.
        \item[(ii)] $X' \cap \Delta(\sigma_m)$ is not a simplex.
    \end{enumerate}
    
    In Case (i), by Lemma~\ref{lem:inductivestep}, we obtain 
    \[L(X)=L(X') \text{ and } M(X)=M_{\prec}(X)=M_{\prec'}(X')\ge M(X')\]
    On the other hand, we have $L(X') \le M(X')$ by Theorem~\ref{thm:main}, and hence the assumption $L(X)=M(X)$ implies 
    \[L(X')=L(X)=M(X)=M_{\prec}(X)=M_{\prec'}(X')=M(X').\]
    Since $\prec'$ is a weak shelling of $X'$ where $M_{\prec'}(X')=M(X')$ and $L(X')=M(X')$ hold, we can apply the induction hypothesis to $X'$.
    By applying Theorem~\ref{thm:mvseq} to the exact sequence \eqref{mvseq}, we have that the inclusion map $\iota: X'[W] \hookrightarrow X[W]$ induces an isomorphism $\iota_i^*: \tilde{H}_i(X'[W]) \to \tilde{H}_i(X[W])$ for every $W \subset V$ and for every integer $i$.
    Then this shows that $\tilde{\beta}_{k-1}(X[W])=\tilde{\beta}_{k-1}(X'[W]) \leq 1$ for every $W \subset V$.
    In addition, since $\tilde{H}_{k-1}(X'[W]) \neq 0$ for some $W \subset V$, then we also have $\tilde{H}_{k-1}(X[W]) \neq 0$.
    Now, suppose $\tilde{H}_{k-1}(X[W]) \neq 0$.
    Then we have $\tilde{H}_{k-1}(X'[W]) \neq 0$.
    Thus $X'[W]$ contains an induced subcomplex $Y$ such that $Y$ is an induced cycle if $k=2$ and $Y$ is isomorphic to $\partial \Delta(k+1)$ if $k \ge 3$ where $Y$ corresponds to a generator of $\tilde{H}_{k-1}(X'[W])$.
    Since $X'[W]$ is an induced subcomplex of $X[W]$, $Y$ is also an induced subcomplex of $X[W]$.
    Since the isomorphism $\iota_i^*: \tilde{H}_i(X'[W]) \to \tilde{H}_i(X[W])$ is induced from the inclusion map, this implies that $Y$ also corresponds to a generator of $\tilde{H}_{k-1}(X[W])$.
    Thus, this proves the theorem for Case (i).

    In Case (ii), by Lemma~\ref{lem:inductivestep}, we have
    \[L(X) \le L(X')+1 \text{ and } M(X)=M_{\prec}(X)=M_{\prec'}(X)+1 \ge M(X')+1.\]
    Thus we obtain $L(X)\le L(X')+1 \le M(X')+1 \le M(X)$.
    Then the assumption $L(X) = M(X)$ implies
    \[L(X')=L(X)-1=M(X)-1=M_{\prec}(X)-1=M_{\prec'}(X')=M(X').\]
    Since $\prec'$ is a weak shelling of $X'$ such that $M_{\prec'}(X')=M(X')$ and $L(X')=M(X')$, again, we can apply the induction hypothesis to $X'$.
    We consider the case $L(X) = 2$ and the case $L(X) \geq 3$ separately.
    
    Suppose $L(X)=2$, then it follows that $M(X) = 2$ and $L(X')=M(X')=1$.
    Take $W\subset V$ such that $\tilde{\beta}_1(X[W])\neq0$.
    We will show that $\tilde{\beta}_1(X[W])\leq 1$ and $X[W]$ contains an induced cycle, which corresponds to a generator of $\tilde{H}_1(X[W])$.
    By \eqref{mvseq}, we have the following exact sequence.
    \[0\to \tilde{H}_1(X[W]) \xrightarrow{\delta_1} \tilde{H}_0((X'\cap \Delta(\sigma_m))[W]) \rightarrow \tilde{H}_0(X'[W]) \to \cdots.\]
    Since $\delta_1$ is an injective map, we have $\tilde{\beta}_1(X[W]) \le \tilde{\beta}_0((X'\cap \Delta(\sigma_m))[W])$.
    Since $\prec$ is a weak shelling, $(X'\cap \Delta(\sigma_m))[W]$ has at most $2$ connected components, and hence we have $\tilde{\beta}_1(X[W]) \le \tilde{\beta}_0((X'\cap \Delta(\sigma_m))[W]) \le 1$.
    
    Now take $\alpha \in C_1(X[W])$ and a positive integer $s$ such that $[\alpha]$ is a generator of $\tilde{H}_1(X[W])$ and $s$ is the minimum where $\alpha=\sum_{i=1}^{s}\alpha_i$ for some $1$-dimensional faces $\alpha_i$.
    We claim that the vertex subset $\cup_{i=1}^s\alpha_i$ induces a cycle in $X[W]$.
    Let $G$ be the graph on $\cup_{i=1}^s\alpha_i$ where the edges are $\alpha_1,\dots,\alpha_s$.
    Since $\partial_1(\alpha)=0$, every vertex in $G$ has an even degree.
    It is a well-known fact that a graph where all vertices have even degree contains a cycle, thus $G$ can be decomposed into edge-disjoint union of cycles.
    Note that since $[\alpha]$ is a generator of $\tilde{H}_1(X[W])$, $\alpha$ is not an $1$-boundary of $X[W]$.
    If $G$ contains at least two edge-disjoint cycles, then at least one of them is not an $1$-boundary of $X[W]$: otherwise, $G$ itself is an $1$-boundary of $X[W]$, which is a contradiction to the assumption that $[\alpha]$ is a generator of $\tilde{H}_1(X[W])$.
    Such a cycle, say $\alpha'$, is a generator of $\tilde{H}_1(X[W])$. On the other hand, we can express $\alpha'$ as $\alpha'=\sum_{i=1}^{s'}\alpha_i'$ with $s'<s$, but this is a contradiction to the minimality of $s$.
    Thus $G$ must be a cycle.
    Now, suppose $\cup_{i=1}^s\alpha_i$ does not induce a cycle in $X[W]$. Then there is a chord of the cycle $G$ in $X[W]$.
    Then this chord divides the cycle $G$ into two smaller cycles $C_1$, $C_2$ such that $C_1$ and $C_2$ share only one edge, which is a chord of $G$.
    If both $C_1$ and $C_2$ are $1$-boundaries of $X[W]$, then $G$ is also an $1$-boundary of $X[W]$, which contradicts that $\alpha$ is a generator of $\tilde{H}_1(X[W])$.
    If one of $C_1$ and $C_2$ is not an $1$-boundary of $X[W]$, then such cycle will be a generator of $\tilde{H}_1(X[W])$ with a smaller size, and hence this also contradicts the choice of $\alpha$.
    Thus, we obtain that $G$ is an induced cycle of $X[W]$.
    Therefore, $X[W]$ contains an induced cycle, which corresponds to a generator of $\tilde{H}_1(X[W])$.

    Now we assume $L(X)=k \ge 3$. Then $L(X')=k-1$.
    Take $W \subset V$ such that $\tilde{\beta}_{k-1}(X[W]) \neq 0$.
    We will show that $\tilde{\beta}_{k-1}(X[W]) \le 1$ and $X[W]$ contains an induced subcomplex isomorphic to $\partial \Delta(k+1)$, which corresponds to a generator of $\tilde{\beta}_{k-1}(X[W])$.
    By \eqref{mvseq}, we have the following exact sequence.
    \[0\to \tilde{H}_{k-1}(X[W]) \xrightarrow{\delta_{k-1}} \tilde{H}_{k-2}((X'\cap \Delta(\sigma_m))[W]) \rightarrow \tilde{H}_{k-2}(X'[W]) \to \cdots.\]
    Since $\delta_{k-1}$ is an injective map, we have $\tilde{\beta}_{k-1}(X[W]) \leq \tilde{\beta}_{k-2}((X'\cap \Delta(\sigma_m))[W])$.
    Since $\tilde{\beta}_{k-1}(X[W]) \neq 0$, we have $\tilde{\beta}_{k-2}((X'\cap \Delta(\sigma_m))[W]) \neq 0$.
    Note that $(X'\cap \Delta(\sigma_m))[W]$ is an induced subcomplex of $X'$.
    Then by the induction hypothesis, we have $\tilde{\beta}_{k-1}(X[W]) \leq \tilde{\beta}_{k-2}((X'\cap \Delta(\sigma_m))[W])\leq 1$.
    In addition, there is an induced subcomplex $Y(k)$ of $(X'\cap \Delta(\sigma_m))[W]$ that corresponds to a generator of $\tilde{H}_{k-2}((X'\cap \Delta(\sigma_m))[W])$ and $Y(k)$ is an induced cycle if $k=3$ and $Y(k) \cong \partial\Delta(k)$ if $k \ge 4$.
    From the definition of weak shelling, $(X'\cap \Delta(\sigma_m))[W]$ cannot contain an induced cycle of length at least $4$.
    Thus we can conclude that $Y(k) \cong \partial\Delta(k)$ for $k \ge 3$.
    Let $v_1,\dots,v_k$ be the vertices of $Y(k)$, and let $g_i=\{v_1,\dots,v_k\}\setminus\{v_i\}$ for each $i \in [k]$. Then $g_1,\dots,g_k$ are the facets of $Y(k)$.
    
    Now take a generator $[\alpha]$ of $\tilde{H}_{k-1}(X[W])$, where $\alpha \in C_{k-1}(X[W])$.
    Then $\delta_{k-1}([\alpha])$ is a generator of $\tilde{H}_{k-2}((X'\cap \Delta(\sigma_m))[W])$.
    Since $X[W]=X'[W] \cup \Delta(\sigma_m)[W]$, $\alpha=\alpha_1+\alpha_2$ for some $\alpha_1 \in C_{k-1}(X'[W])$, $\alpha_2 \in C_{k-1}(\Delta(\sigma_m)[W])$.
    Then we have $\delta_{k-1}([\alpha])=[\partial_{k-1}(\alpha_1)]=[\partial_{k-1}(\alpha_2)]$ by Theorem~\ref{thm:mvseq}.
    
    Let $g=\sum_{i=1}^{k}g_i \in C_{k-2}((X'\cap \Delta(\sigma_m))[W])$.
    We claim that $g \in \partial_{k-1}(C_{k-1}(X'[W]))$.
    Note that $[g]$ is a generator of $\tilde{H}_{k-2}((X'\cap \Delta(\sigma_m))[W])$.
    Since we know $\tilde{\beta}_{k-2}((X'\cap \Delta(\sigma_m))[W])=1$, it must be $[g]=[\partial_{k-1}(\alpha_1)]$.
    This implies that $g-\partial_{k-1}(\alpha_1)=\partial_{k-1}(\gamma)$ for some $\gamma \in C_{k-1}((X'\cap \Delta(\sigma_m)[W]))$.
    Thus $g=\partial_{k-1}(\alpha_1+\gamma)$ and $\alpha_1+\gamma \in C_{k-1}(X'[W])$.

    Now we will show that there is a vertex $w \in W\setminus \sigma_m$ such that $g_i \cup\{w\} \in X'[W]$ for each $i \in [k]$.
    Then $\{v_1,\dots,v_k,w\}$ induces a subcomplex isomorphic to $\partial\Delta(k+1)$ in $X[W]$.
    By Lemma~\ref{lem:inducedorder}, we have a weak shelling $\prec'_W$ of $X'[W]$ such that $M_{\prec'_W}(X'[W]) \le M_{\prec'}(X')=k-1$.
    Then by Theorem~\ref{thm:main}, we have $L(X'[W]) \le M_{\prec'_W}(X'[W]) \le k-1$.
    Since $X'[W]$ contains an induced subcomplex $(X'\cap \Delta(\sigma_m))[W]$ such that $\tilde{H}_{k-2}((X'\cap \Delta(\sigma_m))[W]) \neq 0$, we obtain that $L(X'[W])=M_{\prec'_W}(X'[W])=k-1$.
    Then, by Lemma~\ref{lem:commonvertex}, there is a vertex $w \in W$ such that $g_i\cup\{w\} \in X'[W]$ for each $i \in [k]$.
    Now, we need to show $w \not\in \sigma_m$.
    Assume $w \in \sigma_m$.
    Then $g_i \cup \{w\} \in (X'\cap \Delta(\sigma_m))[W]$ for each $i \in [k]$.
    Since we know $\partial_{k-1}(\sum_{i=1}^k (g_i\cup\{w\}))=g$, $[g]$ cannot be a generator of $\tilde{H}_{k-2}((X'\cap \Delta(\sigma_m))[W]$, which is a contradiction.
    Therefore, there is a vertex $w \in W \setminus \sigma_m$ such that $g_i \cup\{w\} \in X'[W]$ for each $i \in [k]$ and $\{v_1,\dots,v_k,w\}$ induces a subcomplex isomorphic to $\partial\Delta(k+1)$ in $X[W]$.

    Let $\delta=\{v_1,\dots,v_k\}$ and $\tilde{g}=\sum_{i=1}^k (g_i\cup\{w\})+\delta \in C_{k-1}(X[W])$.
    Note that $\partial_{k-1}(\tilde{g})=0$.
    Now we show that $[\tilde{g}]$ is a generator of $\tilde{H}_{k-1}(X[W])$, and this will complete the proof.
    Suppose $[\tilde{g}]$ is not a generator of $\tilde{H}_{k-1}(X[W])$.
    Then there is a $k$-chain $\tau$ of $X[W]$ such that $\partial_k(\tau)=\tilde{g}$.
    By Lemma~\ref{lem:inducedorder}, we have a weak shelling $\prec_W$ of $X[W]$ such that $M_{\prec_W}(X[W]) \le M_{\prec}(X) \le k$.
    By Theorem~\ref{thm:main}, we have $L(X[W]) \le M_{\prec_W}(X[W])=k$.
    Since we know $\tilde{H}_{k-1}(X[W]) \neq 0$, we obtain $L(X[W])=M_{\prec_W}(X[W])=k$.
    Then by Lemma~\ref{lem:commonvertex}, there is a vertex $u \in W$ such that $g_i\cup\{w,u\},\delta\cup\{u\} \in X[W]$ for each $i \in [k]$.
    Since $\delta$ is not a face of $X'[W]$, $\delta \not\subset \sigma_i \cap W$ for all $i<m$.
    Thus we have $\delta \cup\{u\} \subset \sigma_m\cap W$.
    Since $w\notin\sigma_m$, it must be $g_i\cup\{w,u\} \not\subset \sigma_m$ for each $i \in [k]$.
    Thus $g_i \cup \{w,u\} \in X'[W]$ for each $i \in [k]$.
    Then $(X'\cap\Delta(\sigma_m))[W]$ contains $g_i\cup\{u\}$ for each $i \in [k]$.
    Since $\partial_{k-1}(\sum_{i=1}^k(g_i\cup\{u\}))=g$, this is a contradiction that $[g]$ is a generator of $\tilde{H}_{k-2}((X'\cap\sigma_m)[W])$.
    Therefore, $[\tilde{g}]$ is a generator of $\tilde{H}_{k-1}(X[W])$ as we wanted.
\end{proof}

\begin{remark}
    When a simplicial complex $X$ is pure and strongly connected, then $X$ always admits a weak shelling $\prec$ where $M_\prec(X) = M(X)$.
    We also notice that there is a pure and strongly connected complex $X$ with $L(X) = M(X)$, thus satisfies Theorem~\ref{thm:eqaulity} (3), but has a more complicated structure than expected: see \cite[Section 4]{MR1838920}.
    
    Let $X$ be a simplicial complex on $\{1,2,\ldots,8\}$ with seven facets \[\{1235, 1356, 1346, 1467, 1247, 3468, 2348\}.\]
    Clearly, $X$ is $3$-pure and the ordering \[1235 \prec 1356 \prec 1346 \prec 1467 \prec 1247 \prec 3468 \prec 2348\]
    confirms that $X$ is strongly connected.
    It is easy to check that $M(X) = M_\prec(X) = 3 = L(X)$.
    However, this does not contain the join of a vertex and the boundary of a simplex on $4$ vertices as an induced subcomplex: the only induced subcomplex that is isomorphic to the boundary of a simplex on $4$ vertices is $X[\{1,2,3,4\}]$, but none of $1245, 1236, 2347, 1238$ is a face of $X$.
\end{remark}

\section{Weak Eisenbud--Goto inequality for Stanley--Reisner ideals}\label{sec:rmk}
Terai showed in \cite{MR1838920} that if the Stanley--Reisner complex $X$ of a square-free monomial ideal $I$ is pure and strongly connected, then 
\[
\reg(I) = L(X)+1 \le \deg(k[X]) - \codim(k[X])+1.
\]
It has the same form as the well-known Eisenbud--Goto conjecture, which gives an upper bound for the Castelnuovo--Mumford regularity \cite{MR741934}.
We remark that finding a slightly weaker upper bound for the Castelnuovo--Mumford regularity of the form $(\deg - \codim + 1 + c)$ for some $c$ is still meaningful, see \cite{MR2601633}, \cite{MR1620706}, \cite{MR1774091} for more discussion and applications. 
In this section, we relate our main result in \Cref{Sec:main} with such a direction.

Let $X$ be a simplicial complex with a linear order of the facets $\prec: \sigma_1 \prec \dots \prec \sigma_m$. 
We define the \textit{connectivity} of $\sigma_j$ with respect to $\prec$ by the maximum number of vertices in $\sigma_j$ that are shared with any of the facets $\sigma_i$ for $i<j$. 
We denote this quantity by $\conn_{\prec} \sigma_j$, and it is given as follows: $\conn_{\prec} \sigma_j := \max\{|\sigma_i \cap \sigma_j|:i<j\}$ for $j = 2, \dots, m$.

\begin{lemma}\label{lem:conn}
    Let $X$ be a simplicial complex with a linear order of the facets $\prec: \sigma_1 \prec \dots \prec \sigma_m$.
    Then,
    \begin{align}\label{WEGinequ}
        M_{\prec}(X) \le m - |V(X)|+ |\sigma_1|+\sum_{j=2}^m (\dim\sigma_j - \conn_{\prec} \sigma_j).
    \end{align}
    \end{lemma}
\begin{proof}[Sketch of the proof]
The idea is essentially the same as the proof of Lemma~\ref{lem:inductivestep}.

Let the right-hand side of the inequality \eqref{WEGinequ} denote by $N_{\prec}(X)$.
We use the induction on the number of facets $m$.
It is obvious when $m=1$.
Assume that $m \ge 2$ and $M_{\prec}(X') \le N_{\prec}(X')$ where $X' = \cup_{i=1}^{m-1} \Delta(\sigma_i)$.
We have
$$N_{\prec}(X) = N_{\prec}(X')+|V(X'\cap \Delta(\sigma_m))|-\conn_{\prec}\sigma_m.$$

If $X' \cap \Delta(\sigma_m)$ is a simplex, then
\[M_{\prec}(X)=M_{\prec}(X') \le N_{\prec}(X')=N_{\prec}(X).\]
If $X' \cap \Delta(\sigma_m)$ is not a simplex, then
$$M_{\prec}(X) \le M_{\prec}(X') + 1 \le N_{\prec}(X') + 1 \le N_{\prec}(X).$$
Therefore, we have $M_{\prec}(X) \le N_{\prec}(X)$ in both cases.
\end{proof}

Let $\alpha(X)$ denote the number of facets in $X$ whose dimension is strictly less than the dimension of $X$, and let 
$\gamma_{\prec}(X)=\sum_{j=2}^{m}(\dim \sigma_j -\conn_{\prec} \sigma_j)$ for an ordering $\prec:\sigma_1 \prec \cdots \prec \sigma_m$ of facets of $X$ and $\gamma(X)$ be the minimum value of $\gamma_{\prec}(X)$ over all facet orderings $\prec$ of $X$.

\begin{corollary}[Weak Eisenbud--Goto inequality for Stanley--Reisner ideals]\label{cor:weakEG}
Let $X$ be the Stanley--Reisner complex of a square-free monomial ideal $I$.
Then,
\[
\reg(I)\le \deg(k[X]) - \codim(k[X]) + 1 + \alpha(X) + \gamma(X)
\]
\end{corollary}
\begin{proof}
Let $m$ be the number of facets of $X$.
Recall that $L(X) = \reg(I)-1$, $\deg(k[X])$ equals the number of facets whose dimension is $\dim X$, and $\codim(k[X]) = |V(X)| - \dim X - 1$.
In particular, $\alpha( X)=m-\deg(k[ X])$.
Let $\prec: \sigma_1 \prec \cdots \prec \sigma_m$ be a linear ordering of the facets of $X$ such that $\gamma_{\prec}(X)=\gamma(X)$.
Since we have $|\sigma_1| \le \dim X +1$, Lemma \ref{lem:conn} implies
$$L( X) \le \deg(k[ X]) - \codim(k[ X]) + \alpha( X) + \gamma( X)$$ as desired.
\end{proof}

Note that $X$ is pure and strongly connected if and only if $\alpha(X)=\gamma(X)=0$. Thus Corollary \ref{cor:weakEG} generalizes Terai's results in \cite{MR1838920}. 

\begin{example}
    Let $J = (x_0^2,x_1^2,x_0x_2+x_1x_3)$ be an ideal in $S = k[x_0,\dots,x_3]$ and $I\subset T = k[z_{00}, z_{01}, z_{10}, z_{11}, z_{20}, z_{21}, z_{30}]$ be the polarization of the generic initial ideal of $J$ with respect to the reverse lexicographical ordering.

    The Stanley--Reisner complex $ X$ of $I$ consists of $4$ facets $\sigma_1 = \{z_{00},z_{11},z_{20}\}$, $\sigma_2 = \{z_{00},z_{11},z_{21},z_{30}\}$, $\sigma_3 = \{z_{01},z_{11},z_{20},z_{21},z_{30}\}$, $\sigma_4 = \{z_{01},z_{10},z_{20},z_{21},z_{30}\}$.
    
    One can compute numerical information of the ideals as follows:
    \begin{itemize}
        \item $\reg J = \reg I = 3$, i.e., $L( X) = 2$.
        \item $\deg S/J = \deg T/I = 2$.
        \item $\codim S/J = \codim T/I = 2$.
    \end{itemize}   
    The ideal $J$ does not satisfy the Eisenbud--Goto inequality: $$\deg S/J - \codim S/J + 1 = 1 < 3 = \reg J.$$
    On the other hand, by Corollary \ref{cor:weakEG}, we have
    \[\reg I \le \deg(k[ X]) - \codim(k[ X]) + 1 + \alpha( X) + \gamma( X) =3 \]
    with the facet ordering $\sigma_4  \prec \sigma_3 \prec \sigma_2 \prec \sigma_1$.
    This gives the tight bound on the regularity of $I$.
\end{example}

The connectivity of a simplicial complex $X$ is also related to the regularity of a coordinate subspace arrangement, by considering the weighted Euler characteristic of a weighted graph associated with $X$. 
We discuss a reformulation of Corollary~\ref{cor:weakEG} in this direction.
We recall the setting of \cite[Section 5]{MR2275024}.

For a simplicial complex $X$ with the facets $\sigma_1,\ldots,\sigma_m$, we define a complete graph $G=(V(G), E(G))$ with weights on the vertices and the edges as follows: 
Let $V(G)=[m]$ with weight $w(i) = |\sigma_i|$ for $i=1,\dots,m$. 
Also, let $E(G)={[m] \choose 2}$ with weight $w(\{i,j\})=|\sigma_i\cap\sigma_j|$ for $1\le i < j \le m$.
The \textit{weighted Euler characteristic} $\chi_w(G)$ of $G$ is the sum of the weights of the vertices in $G$ minus the sum of the weights of the edges in $G$, i.e., $\chi_w(G):= \sum_{v\in V(G)} w(v) - \sum_{e\in E(G)} w(e)$.

\begin{corollary}\label{cor:Weakeular}
    Let $X$ be a simplicial complex with a facet ordering $\prec:\sigma_1 \prec \cdots \prec \sigma_m$.
    Suppose $G$ is the weighted graph of $X$ constructed as above.
    Then there is a spanning tree $F$ of $G$ such that
    \[
    N_{\prec}(X)= \chi_w(F) - |V(X)| + 1.
    \]
\end{corollary}
\begin{proof}
We inductively construct a tree $F$ on $[m]$.
Start with a graph with no edges, and for each $i \in \{2,3,\ldots,m\}$ add an edge that connects the vertices $i$ and $j$ with weight $|\sigma_j \cap \sigma_i|$ where $j <i$ is the minimum index among those with $|\sigma_j \cap \sigma_i| = \max_{t < i}|\sigma_t \cap \sigma_i|$. This obviously gives a spanning tree of $G$.

Now the statement immediately follows from
    \begin{align*}
        N_{\prec}(X) + |V(X)| & = m +|\sigma_1| +\sum_{j = 2}^{m} (\dim\sigma_j-\conn_<\sigma_j) \\
         & = \sum_{i = 1}^m |\sigma_i| - \sum_{j = 2}^m \max_{k<j}\{\sigma_k\cap\sigma_j\} + 1 \\
         &= \sum_{v\in V(F)} w(v) - \sum_{e\in E(F)} w(e) + 1 \\
         & = \chi_w(F) +1.\qedhere
    \end{align*}
\end{proof}

The authors of \cite{MR2275024} showed that the Stanley--Reisner ideal $I_X$ of $X$ is $2$-regular if and only if there is a spanning tree $F$ of $G$ such that $\chi_w(F)=|V(X)|$.
Corollary~\ref{cor:Weakeular} describes what happens for ideals with higher regularity. 

We summarize the series of inequalities as follows:
$$L(X)\le M(X) \le N_{\prec}(X).$$
Note again that, if $X$ is a $1$-Leray complex, all quantities in the inequality are equal to each other.
We conclude this section with an example where the equalities do not hold.

\begin{example}
    Let $X=\Delta(\{1,2,4\})\cup \Delta(\{1,3,5\})\cup \Delta(\{2,3,6\})\cup \Delta(\{4,5,6\})$ be the complex in \Cref{eg:SRcorr}.
    In this case, we have $$L(X) = 2 < M(X) = 3 < N_\prec(X) = 4$$
    for any ordering $\prec$ of the facets.  
\end{example}

\section*{Acknowledgement}
J.~Jung was supported by the Institute for Basic Science (IBS-R032-D1-2023-a00).
J.~Kim was supported by the Institute for Basic Science (IBS-R029-Y5).
M.~Kim was supported by Basic Science Research Program of the National Research Foundation of Korea (NRF) funded by the Ministry of Education (NRF-2022R1F1A1063424) and by GIST Research Project grant funded by the GIST in 2023. Y. ~Kim was supported by the Basic Science Program of the NRF of Korea (NRF-2022R1C1C1010052).

\bibliographystyle{acm}
\bibliography{reference.bib}

\end{document}